\documentclass{amsart}

\usepackage[skip=5pt plus1pt, indent=20pt]{parskip}

\usepackage{graphicx} 

\DeclareMathOperator{\Norm}{Norm}

\usepackage[utf8]{inputenc}
\usepackage{amssymb} 
\usepackage{amsmath} 
\usepackage{amscd}
\usepackage{amsbsy}
\usepackage{comment}
\usepackage[matrix,arrow]{xy}
\usepackage{hyperref}
\usepackage{float}
\usepackage[utf8]{inputenc}
\usepackage{xcolor}
\usepackage{enumerate}

\newcommand{\cO}{\mathcal{O}}
\newcommand{\Z}{\mathbb{Z}}
\newcommand{\Q}{\mathbb{Q}}

\begin{document}

\newtheorem{theorem}{Theorem}
\newtheorem{lemma}{Lemma}[section]
\newtheorem{proposition}[lemma]{Proposition}
\newtheorem{algorithm}[lemma]{Algorithm}
\newtheorem{corollary}[lemma]{Corollary}
\newtheorem{conjecture}{Conjecture}
\newtheorem{example}[lemma]{Example}

\theoremstyle{definition}
\newtheorem{definition}[theorem]{Definition}

\theoremstyle{remark}
\newtheorem{remark}[theorem]{Remark}

\newtheorem{acknowledgment}{Acknowledgement}

\title[]{On generalised Pythagorean triples over number fields}

\author{Pedro-Jos\'{e} Cazorla Garc\'{i}a}

\address{Departamento de Matem\'atica Aplicada, ICAI, Universidad Pontificia Comillas, Madrid, 28015, Spain}
\email{pjcazorla@comillas.edu}

\date{\today}

\begin{abstract}
    Generalised Pythagorean triples are integer tuples $(x,y,z)$ satisfying the equation $E_{a,b,c}: ax^2+by^2+cz^2=0$. A significant amount of research has been devoted towards understanding generalised Pythagorean triples and, in particular, we can now determine whether $E_{a,b,c}$ has solutions and find them in a computationally effective manner. 
    
    In this paper, we consider an extension of generalised Pythagorean triples to number fields $K$. In particular, we survey and extend the existing results over $\Q$ for determining if $E_{a,b,c}$ has solutions over number fields and if so, to find and parameterise them, as well as to find a minimal solution. Throughout the text, we incorporate numerous examples to make our results accessible to all researchers interested in the topic of generalised Pythagorean triples.
\end{abstract}


\maketitle

\section{Introduction}
\subsection{Background}
\label{Sec:background}
A Pythagorean triple is a tuple of three positive integers $(x,y,z)$ satisfying
\begin{equation}
    \label{eqn:pythagorean}
x^2 + y^2 = z^2.
\end{equation}
The oldest known record of Pythagorean triples comes from Plimpton 322, a Babylonian tablet dating back to c.~1800 BC (see \cite{history} for a complete discussion of Plimpton 322). In this tablet, there is a list of $15$ Pythagorean triples with $x \le 12709$. Note that there are several mistakes in Plimpton 322, including an incorrect solution with $x = 25921$. While this is the oldest instance of Pythagorean triples, they were studied by many people in Antiquity, including Pythagoras himself.

However, a complete solution of \eqref{eqn:pythagorean} was not attained until Euclid (c.~300 BC) \cite[Book X, Proposition XXIX]{Euclid} showed that all solutions $(x,y,z) \in \Z^3$ to \eqref{eqn:pythagorean} with $\gcd(x,y,z) = 1$  could be parameterised by the following formulae:
\begin{equation}
    \label{eqn:parameterisation}
    \begin{cases}
    x = m^2-n^2, \\
    y = 2mn, \\
    z = m^2+n^2,
\end{cases}
\end{equation}
where $m, n \in \Z$ with $m > n > 0$. It is clear that, if $(\alpha, \beta, \gamma) \in \Z^3$ is a solution to \eqref{eqn:pythagorean}, so is $(d\alpha, d\beta, d\gamma) \in \Z^3$ for any $d \in \Z$. Consequently, it is enough to consider solutions $(x,y,z) \in \Z^3$ with $\gcd(x,y,z) = 1$, which we shall call \emph{primitive solutions}.

Euclid's work showed that \eqref{eqn:pythagorean} had infinitely many primitive solutions and that all of them could be found via an explicit parametric formula. Therefore, the Pythagorean equation was essentially completely solved in ancient times.

However, it was not until much later that generalisations of \eqref{eqn:pythagorean} were considered. Indeed, at the end of the 18th century, Legendre studied the \textbf{generalised Pythagorean equation}:
\begin{equation}
    \label{eqn:generalisedPythagorean}
    ax^2 + by^2 + cz^2 = 0, 
\end{equation}
where $a$, $b$ and $c$ are fixed squarefree integers with $\gcd(a, b, c) = 1$, and we are interested in primitive solutions $(x, y, z) \in \Z^3$ where $x$, $y$ and $z$ are not all zero. The amount of literature on the generalised Pythagorean equation is vast, and we now have a good understanding of solutions to \eqref{eqn:generalisedPythagorean}.

Firstly, in 1785, Legendre proved that \eqref{eqn:generalisedPythagorean} could have either infinitely many primitive solutions or no solutions, depending on whether certain conditions on $a$, $b$ and $c$ were satisfied.

In addition, he introduced an algorithmic procedure to compute a particular solution to \eqref{eqn:generalisedPythagorean} if there was one. This algorithm essentially proceeds by descent, replacing \eqref{eqn:generalisedPythagorean} by an equation with smaller values of $a$, $b$ and $c$ until a solution can be found quickly by inspection. We refer the reader to \cite[IV.3.3]{Smart} for an in-depth description of Legendre's descent method.

We remark that we now have significantly more efficient algorithms to compute a solution to \eqref{eqn:generalisedPythagorean}. For example, we highlight the work of Cremona and Rusin \cite{Cremona}, which uses the Lenstra-Lenstra-Lovász (LLL) lattice reduction algorithm \cite{LLL} and avoids the factorisation of the integers $a$, $b$ and $c$, as well as the paper by Simon \cite{Simon}, which considers matrices of ternary quadratic forms and makes use of the LLL algorithm to reduce them. This last method is implemented in the computer algebra system \texttt{Magma} \cite{Magma} to solve \eqref{eqn:generalisedPythagorean} over $\Q$.

Even prior to the development of these algorithms, Mordell \cite[Chapter 7, Theorems 4 and 5]{Mordellbook} found that, if a solution to \eqref{eqn:generalisedPythagorean} exists, then there is an explicit parameterisation similar to \eqref{eqn:parameterisation} for all integer solutions, which requires a particular solution $(\alpha_0, \beta_0, \gamma_0) \in \Z^3$. Therefore, the problem of finding all solutions to \eqref{eqn:generalisedPythagorean} is as difficult as that of finding one particular solution.

Inside the set of all solutions, there are some which are ``minimal'' in some sense. Indeed, before the aforementioned work by Mordell,  Hölzer \cite{Holzer} had shown that, if \eqref{eqn:generalisedPythagorean} has solutions, there exists a solution $(x, y, z) \in \Z^3$ satisfying
\begin{equation}
    \label{eqn:Holzercondition}
|x| \le \sqrt{|bc|}, \quad |y| \le \sqrt{|ac|} \quad \text{ and }  |z| \le \sqrt{|ab|}.
\end{equation}
Such a solution is called \textbf{Hölzer-reduced}. A different proof of Hölzer's theorem was later given by Mordell \cite{Mordell}, which was subsequently turned into an algorithm for computing a Hölzer-reduced solution by Cremona and Rusin \cite{Cremona}.

To summarise, the generalised Pythagorean equation \eqref{eqn:generalisedPythagorean} has been vastly studied throughout history, and, in particular, we know the following:
\begin{itemize}
    \item How to determine whether it has solutions or not.
    \item An efficient algorithm to compute a solution $(\alpha_0,\beta_0,\gamma_0) \in \Z^3$, if there is one.
    \item A set of parametric formulae to generate all solutions from a particular solution $(\alpha_0,\beta_0,\gamma_0) \in \Z^3$.
    \item An algorithm to compute Hölzer-reduced solutions from a particular solution $(\alpha_0,\beta_0,\gamma_0)  \in \Z^3$.
\end{itemize}

\subsection{Summary of contents}
In this paper, we consider the \textbf{generalised Py\-tha\-go\-rean equation over number fields}. For this, we let $K$ be a fixed number field with ring of integers $\cO_K$, and we consider solutions $(x,y, z) \in K^3$ to the equation
\begin{equation}
    \label{eqn:main}
    ax^2 + by^2 + cz^2 = 0, \quad a, b, c \in \cO_K \setminus\{0\},
\end{equation}
where the ideals $a\cO_K$, $b\cO_K$ and $c\cO_K$ have no squares of any principal ideals in their prime factorisations (this is the analogue of being squarefree rational integers). 

To see that this condition is well defined, we note that, since $\cO_K$ is a Dedekind domain, every ideal $I$ can be uniquely written as a product of powers of prime ideals (up to reordering), in the form \[I = \mathfrak{p}_1^{\alpha_1}\dots \mathfrak{p}_n^{\alpha_n}.\]
Consider the set $D_I$ of all ideals dividing $I$:
\[D_I = \{\mathfrak{p}_1^{i_1}\dots \mathfrak{p}_n^{i_n} \mid 0\le i_1\le \alpha_1, \dots, 0 \le i_n\le \alpha_n\}.\]
If the set $D_I$ does not contain any squares of principal ideals, we say that the ideal $I$ has no squares of principal ideals in its prime factorisation.

We note that \eqref{eqn:main} will always have the \emph{trivial solution} $(x,y,z) = (0,0,0)$. For this reason, we shall call solutions with $x$, $y$ and $z$ not all zero \emph{non-trivial solutions}, and we will devote our work to finding them. If the class number of $K$ is greater than $1$, there is no notion of the greatest common divisor and, consequently, we cannot a priori generalise the notion of primitive solutions.

In addition, we remark that, when working over number fields, it it not always possible to look for solutions where $x, y, z \in \cO_K$, and so we will be interested in solutions where $x, y, z \in K$. In certain instances, such as if $K$ has class number one, it will be possible to consider solutions where $x, y$ and $z$ are algebraic integers and we will explain this along the text.

It is surprising to note the scarcity of literature on generalisations of the Pythagorean equation over number fields. Virtually all existing works (see \cite{Hemer, LealRuperto, VargasSantos}, for example) focus on the case where $K$ is an Euclidean imaginary quadratic field, where they show sufficient and necessary conditions for solvability and extend Hölzer's Theorem.

In this work, we attempt to take a systematic approach to generalise the results known over $\Q$ to arbitrary number fields in as much generality as possible. To the best of our knowledge, this is the first paper dealing with solutions of generalised Pythagorean equations in this generality, and we hope that it can be useful for researchers interested in the generalised Pythagorean equation. For this purpose, we briefly outline the main results of this paper.


Firstly, we have found a set of necessary and sufficient conditions (Corollary \ref{cor:conditions}) which, if satisfied, guarantee the existence of solutions to \eqref{eqn:main} over $K$. These conditions depend only on the coefficients $a$, $b$ and $c$, as well as on the number field $K$ under consideration. 

If these conditions are fulfilled, there are infinitely many solutions, parameterised by the formulae in Proposition \ref{prop:parameterisation}, which require the use of a particular solution of \eqref{eqn:main}. In order to find one solution in a computationally effective manner, we have generalised Legendre's descent algorithm in Algorithm \ref{alg:Legendre}.

Finally, if $K$ is an Euclidean imaginary field, Díaz--Vargas and Vargas de los Santos (\cite{VargasSantos}) extended the notion of Hölzer-reducedness to solutions of \eqref{eqn:main} over $K$. We shall build upon their work to develop an algorithm (Algorithm \ref{alg:ReduceSolution}) that finds a reduced solution to \eqref{eqn:main}.


Our aim throughout the article is to make the research as accessible as possible to all researchers interested in the explicit resolution of \eqref{eqn:generalisedPythagorean}, which is why we will incorporate numerous examples in the text.

The organisation of this paper is as follows. We remark that each section is independent of the rest, and can be read without loss of continuity.
In Section \ref{Sec:conditions}, we use the Hasse--Minkowski theorem to prove a set of necessary and sufficient conditions for \eqref{eqn:main} to have solutions. In Section \ref{Sec:parametric}, we extend Mordell's work \cite{Mordellbook} to parameterise all solutions $(x,y,z) \in K^3$ to \eqref{eqn:main}. In Section \ref{Sec:algorithm}, we provide a generalisation of Legendre's descent algorithm to find a solution $(x,y,z) \in K^3$ and illustrate it with an example. In Section \ref{Sec:minimal}, we adapt the notion of Hölzer-reduced solutions and develop an algorithm to find a minimal solution from an existing one. Finally, in Section \ref{Sec:conclusions} we briefly discuss potential future lines of work.

\textbf{Acknowledgements: }The author would like to thank Antonella Perucca for suggesting a problem motivating this paper, Martin Orr for his reading and comments of this manuscript, as well as Claus Fieker and Damein Stehlé for kindly sharing their code.

\section{Conditions for solvability}
\label{Sec:conditions}
In this section, we shall determine necessary and sufficient conditions for \eqref{eqn:main} to have solutions over a number field $K$. We recall from Section 1 that we are only interested in \emph{non-trivial solutions} of \eqref{eqn:main}, i.e. solutions $(x,y,z) \in K^3$ where at least one of $x$, $y$ and $z$ is non zero.

For this purpose, we shall use the Hasse--Minkowski theorem (see \cite{Hasse1, Hasse2, Hasse3, Hasse4} for the original proof of the statement or \cite[\S 66]{HasseMinkowski} for a more sucint proof), which is generally known as the \emph{local--global principle}. The version of the theorem that we will need is the following:

\begin{theorem}(Hasse-Minkowski)
    \label{thm:HasseMinkowski}
    Let $n \ge 1$, let $K$ be a number field and let $Q(X_1, \dots, X_n) \in K[X_1, \dots, X_n]$ be a quadratic form defined over $K$. Let $S$ be the set of all places of $K$. For any $\nu \in S$, we will denote the completion of $K$ at the place $\nu$ by $K_\nu$. Then, the equation
    \[Q(X_1, \dots, X_n) = 0,\]
    has a solution $(\alpha_1, \dots, \alpha_n) \in K^n \setminus \{(0, 0, \dots, 0)\}$ if and only if it has a solution $(\alpha_{1,v}, \dots, \alpha_{n,v}) \in K_\nu^n \setminus \{(\tilde{0}, \tilde{0}, \dots, \tilde{0})\}$ for all $\nu \in S$. 
\end{theorem}

Using the Hasse--Minkowski's theorem will allow us to find a set of necessary and sufficient conditions for \eqref{eqn:main} to have non-trivial solutions over a given number field $K$. This is the content of the following corollary, which deals with the case where $a\cO_K$, $b\cO_K$ and $c\cO_K$ have no common prime ideal in their factorisation. This corollary extends existing results over $\Q$ (\cite[Section 2.2]{Cremona}) and over certain Euclidean imaginary quadratic fields (\cite{LealRuperto, VargasSantos}).
\begin{corollary}
    \label{cor:conditions}
    Let $K$ be a number field, with ring of integers $\cO_K$ and assume that there is no prime ideal $\mathfrak{p}$ dividing two of $a\cO_K$, $b\cO_K$ and $c\cO_K$. Then, \eqref{eqn:main} has a non-trivial solution if and only if the following conditions hold:
    \begin{enumerate}
        \item Either $K$ is totally imaginary or for any real embedding $i: K \hookrightarrow \mathbb{R}$, precisely two of $i(a), i(b)$ and $i(c)$ have the same sign.
        \item There is a solution $\alpha \in \cO_K$ to the congruence
        \begin{equation}
            \label{eqn:congruenceA}
        X^2 \equiv -bc \pmod{a\cO_K},
        \end{equation}
        a solution $\beta \in \cO_K$ to the congruence
        \begin{equation}
            \label{eqn:congruenceB}
        X^2 \equiv -ac \pmod{b\cO_K},
        \end{equation}
        and a solution $\gamma \in \cO_K$ to the congruence
        \begin{equation}
            \label{eqn:congruenceC}
        X^2 \equiv -ab \pmod{c\cO_K}.
        \end{equation}
        \item For all prime ideals $\mathfrak{p}$ dividing $2\cO_K$, there is a solution $(\alpha, \beta, \gamma) \in \cO_K^3$ to the congruence
        \[ax^2 + by^2 + cz^2 \equiv 0 \pmod{\mathfrak{p}^{2v+1}},
        \]
        for some $v \ge 1$ and such that either
        \[2a\alpha \not \equiv 0 \pmod{\mathfrak{p}^{v+1}},\]
        or
               \[2b\beta \not \equiv 0 \pmod{\mathfrak{p}^{v+1}},\]
               or
                      \[2c\gamma \not \equiv 0 \pmod{\mathfrak{p}^{v+1}}.\]
    \end{enumerate}
\end{corollary}

\begin{proof}
    Firstly, we suppose that there is a solution $(\alpha, \beta, \gamma) \in K^3$ to \eqref{eqn:main} and let us see that the three conditions are satisfied. Condition (1) corresponds to the fact that, for any embedding $i$, $(\alpha, \beta, \gamma)$ will satisfy
    \begin{equation*}
    i(a)i(\alpha)^2 + i(b)i(\beta)^2+i(c)i(\gamma)^2 = 0,
    \end{equation*}
    and so $(x,y,z) =(i(\alpha), i(\beta), i(\gamma))$ is a solution to the equation 
    \begin{equation}
        \label{eqn:completion}
    i(a)x^2 + i(b)y^2 + i(c)z^2 = 0,
    \end{equation}
    We note that, since $a,b,c \neq 0$, this equation always has non-trivial solutions over $\mathbb{C}$ and only has non-trivial solutions over $\mathbb{R}$ if precisely two of $i(a)$, $i(b)$ and $i(c)$ have the same sign, whence condition (1) follows.

    Condition (2) is checked simply by reducing \eqref{eqn:main} modulo $a\cO_K$, $b\cO_K$ and $c\cO_K$. Finally, condition (3) follows again by reducing the identity 
    \[a\alpha^2 + b\beta^2+c\gamma^2 = 0,\]
    modulo $\mathfrak{p}^{2v+1}$, where $v \ge 1$ is chosen such that one of $2a\alpha$, $2b\beta$ or $2c\gamma$ do not lie in $\mathfrak{p}^{v+1}$.

    Conversely, we suppose that all three conditions are satisfied and we will show that there exists a non-trivial solution over $K$. By Theorem \ref{thm:HasseMinkowski}, it suffices to show that \eqref{eqn:main} has a non-trivial solution over all local completions, so we shall consider each completion separately.

    Firstly, let $\nu$ be an Archimedean place of $K$ with associated embedding $i$. As mentioned above, condition $(1)$ ensures that \eqref{eqn:completion} has a solution over $K_\nu$, regardless on whether $\nu$ is a complex or a real place.

    Now, let $\mathfrak{p}$ be a prime ideal of $\cO_K$ corresponding to a non-archimedean place $\nu$ of $K$. Firstly, we suppose that $\mathfrak{p} \nmid (2abc)\cO_K$, and let $q = \#\cO_K/\mathfrak{p}$. It is clear that the map $\phi: ({\cO_K}/{\mathfrak{p}})^* \mapsto ({\cO_K}/{\mathfrak{p}})^*$ given by $\phi(t) = at^2$ is two-to-one and thus contains $(q-1)/2$ elements in its image. If we extend $\phi$ to the whole ${\cO_K}/{\mathfrak{p}}$ by setting $\phi(0) = 0$, we conclude that there are $(q+1)/2$ elements in the image of $\phi$.

    By an identical argument, the same is true about the map $\phi': {\cO_K}/{\mathfrak{p}} \mapsto {\cO_K}/{\mathfrak{p}}$ given by $\phi'(s) = -bs^2$. Consequently, the set 
    \[\{(-c-bs^2) \pmod{\mathfrak{p}} \mid s \in \cO_K\}\]
    has $(q+1)/2$ elements, and so does the set
    \[\{at^2 \pmod{\mathfrak{p}} \mid t \in \cO_K\}.\]
    Since there are only $q$ elements in $\cO_K/\mathfrak{p}$, the two sets have an element in common and thus 
    \[a\alpha^2 + b\beta^2 + c \equiv 0\pmod{\mathfrak{p}},\]
    for some $\alpha, \beta \in \cO_K$. This means that $(\alpha, \beta, 1)$ is a solution to the congruence equation
    \[ax^2+by^2+cz^2\equiv 0 \pmod{\mathfrak{p}}.\]
    By Hensel's Lemma (see \cite[Prop. 4.1.37]{Cohen}), this can be lifted to a non-trivial solution over $K_\nu$.

    Suppose now that $\mathfrak{p} \nmid 2\cO_K$ and $\mathfrak{p} \mid (abc)\cO_K$. Without loss of generality, we shall assume that $\mathfrak{p} \mid a\cO_K$. Note that, by the assumptions in the statement of the theorem, this means that $\mathfrak{p}\nmid (bc)\cO_K$.
    
    Therefore, once again by Hensel's Lemma, \eqref{eqn:main} will have a non-trivial solution over $K_\nu$ if and only if there is a solution $(\beta, \gamma) \in \cO_K^2 \setminus \{({0}, {0})\}$ to the congruence equation
    \[b\beta^2 + c\gamma^2 \equiv 0 \pmod{\mathfrak{p}}.\]
    Since, by assumption, there is a solution to \eqref{eqn:congruenceA} and $\mathfrak{p} \nmid bc\cO_K$, this congruence is satisfied. 
    By proceeding in a similar manner, we can show that the same is true if either $\mathfrak{p} \mid b\cO_K$ or if $\mathfrak{p} \mid c\cO_K$ by using \eqref{eqn:congruenceB} and \eqref{eqn:congruenceC}, respectively. 

    Finally, suppose that $\mathfrak{p} \mid 2\cO_K$, and let $f(x,y,z) = ax^2+ by^2+cz^2 \in \Z[x,y,z]$. Then, the multivariate version of Hensel's Lemma (see \cite[Theorem 2.1]{Conrad}) yields that \eqref{eqn:main} has a solution over $K_\nu$ if there is an integer $v \ge 1$ and a solution $(x,y,z) \in \cO_K^3$ to the congruence equation
    \[f(x,y,z) \equiv 0 \pmod{\mathfrak{p}^{2v+1}},\]
    and such that at least one component of $\nabla f(x,y,z)$ is non-zero modulo $\mathfrak{p}^{v+1}$. This corresponds precisely to condition $(3)$ in the corollary.
\end{proof}

\begin{remark}
    Note that the assumption that no prime ideal divides two of $a\cO_K$, $b\cO_K$ and $c\cO_K$ was only used to prove condition (2). In Section \ref{subsec:noncoprime}, we will show how to reduce to Corollary \ref{cor:conditions} if $a\cO_K$, $b\cO_K$ and $c\cO_K$ have common prime factors.
\end{remark}

\begin{example}
    Let $K = \Q(\sqrt{-7})$ and we consider the equation 
    \begin{equation}
        \label{eqn:equationconditions}
    3x^2 + 2y^2 +13z^2 = 0.
    \end{equation}
     We check whether the three conditions in Corollary \ref{cor:conditions} are satisfied. Since $K$ is totally imaginary, condition $(1)$ follows. The congruences \eqref{eqn:congruenceA}, \eqref{eqn:congruenceB} and \eqref{eqn:congruenceC} amount to
    \[X^2 \equiv -26 \pmod{3\cO_K},\]
    \[X^2 \equiv -39 \pmod{2\cO_K},\]
    and 
    \[X^2 \equiv -6 \pmod{13\cO_K},\]
    respectively. The elements $\alpha = 1$, $\beta = 1$ and $\gamma = 5\sqrt{-7}$ satisfy the three congruences and, consequently, condition $(2)$ in Corollary \ref{cor:conditions} holds. 

    Finally, we check condition $(3)$. For this, we note that $2$ has the following factorisation in $\cO_K$:
    \[2 = \frac{1+\sqrt{-7}}{2}\cdot \frac{1-\sqrt{-7}}{2}.\]
    Let $\alpha = ({1+\sqrt{-7}})/{2}$ and $\beta = ({1-\sqrt{-7}})/{2}$. We note that 
    \[3\cdot 1^2 + 2\cdot 0^2 + 13\cdot 1^2 = 16 \equiv 0 \pmod{8},\]
    and, since $\alpha^3 \mid 8$, it is easy to see that 
    \[3\cdot 1^2 + 2\cdot 0^2 + 13\cdot 1^2 = 16 \equiv 0 \pmod{\alpha^3}.\]
    At the same time, we see that 
    \[2\cdot 3 \cdot 1 = 6 \not \equiv 0 \pmod{\alpha^2},\]
    so that condition $(3)$ is satisfied for the ideal $\alpha\cO_K$. An analogous computation shows that it is also satisfied for $\beta\cO_K$, whence Corollary \ref{cor:conditions} shows that \eqref{eqn:equationconditions} has some solutions over $K$. In fact, it is elementary to check that $(x,y,z) = (\sqrt{-7}, 2, 1) \in \cO_K^3$ is a solution to \eqref{eqn:equationconditions}.
\end{example}

\subsection{The non coprime case}
\label{subsec:noncoprime}
In order to complete our treatment of \eqref{eqn:main}, we need to consider the case where the ideals $a\cO_K$, $b\cO_K$ and $c\cO_K$ have a common prime factor $\mathfrak{p}$. Since the situation where $\mathfrak{p} \mid 2\cO_K$ is already covered in Corollary \ref{cor:conditions}, we shall only consider odd primes $\mathfrak{p}$. To do this, we will need the following lemma.

\begin{lemma}
    \label{lemma:complicated}
    Let $K$ be a number field with ring of integers $\cO_K$ and let $v$ be a non-Archimedean odd place such that at least two of $v(a), v(b)$ and $v(c)$ are positive. Without loss of generality, we shall assume that $v(a)\ge v(b) \ge v(c)$.
    
    Let $\mathfrak{p}$ be the corresponding ideal of $\cO_K$, $K_v$ denote the localisation of $K$ at $v$, $\cO_{K_v}$ its ring of integers and $\pi$ its uniformiser. Then
    \begin{enumerate}
        \item If $v(a),v(b),v(c) > 0$, \eqref{eqn:main} has a non-trivial solution over $\cO_{K_v}$ if and only if the equation 
        \[\frac{a}{\pi^{v(c)}}x^2 + \frac{b}{\pi^{{v(c)}}}y^2 + \frac{c}{\pi^{v(c)}}z^2 = 0,\]
        has a solution over $\cO_{K_v}$.
        \item If $v(a) > v(b) > 0$ and $v(c) = 0$, \eqref{eqn:main} has a non-trivial solution over $\cO_{K_v}$ if and only if there is a solution $\alpha \in \cO_K$ to the congruence
        \begin{equation}
            \label{eqn:newcongruence}
        X^2 \equiv -bc 
        \pmod{\mathfrak{p}^{v(a)}}.
        \end{equation}
        \item If $v(a) = v(b) > 0$ and $v(c) = 0$, we write $v(a) = 2v' + v''$. Then, \eqref{eqn:main} has a non-trivial solution over $\cO_{K_v}$ if and only if the equation 
        \begin{equation*}
        \frac{a}{\pi^{v(a)}}x^2 + \frac{b}{\pi^{{v(a)}}}y^2 + {c}{\pi^{v''}}z^2 = 0,
        \end{equation*}
        has a solution over $\cO_{K_v}$.
    \end{enumerate}
\end{lemma}

\begin{proof}
    Case (1) is clear since we are simplifying equation \eqref{eqn:main} by dividing by $\pi^{v(c)}$. Note that, since $v(a)\ge v(b) \ge v(c)$, the new equation has at most two coefficients with positive valuation.

    For case (2), we first suppose that \eqref{eqn:main} has a non-trivial solution $(\alpha, \beta, \gamma) \in \cO_K^3$. Then, we see that the congruence holds in precisely the same manner as in the proof of Corollary \ref{cor:conditions}. Conversely, if the congruence equation holds, we may multiply \eqref{eqn:newcongruence} by $c$ to see that
    \[bc^2 + c\alpha^2 \equiv 0 \pmod{\mathfrak{p}^{v(a)}},\]
    and thus 
    \[v(bc^2 + c\alpha^2) \ge v(a).\]
    Since, by assumption, $v(bc^2) = v(b) < v(a)$ and $v(c) = 0$, this is only possible if 
    \[v(\alpha) = \frac{v(b)}{2} < \frac{v(a)}{2}.\]
    If we let $f(x,y,z) = ax^2+by^2+cz^2$, the previous computation shows that 
    \[v(f(1, c, \alpha)) \ge v(a),\]
    while
    \[2v\left(\left(\frac{\partial f}{\partial z}\right)(1,c,\alpha)\right) = 2v(2c\alpha) < v(a), \]
    and the multivariate version of Hensel's Lemma shows that there is a solution to \eqref{eqn:main}, in the same vein as in the proof of Corollary \ref{cor:conditions}.

    Finally, suppose that \eqref{eqn:main} has a solution $(\alpha, \beta, \gamma) \in \cO_{K_v}$ and let us show that condition (3) holds. If $v(a) = v(b) = 2v'+v''$ and $v(c) = 0$, 
    it follows that $v(\gamma) \ge v' + v''/2$. We may then write $a=\pi^{2v'+v''}a'$ and $b = \pi^{2v'+v''}b'$, where $v(a') = v(b') = 0$. Similarly, we write $\gamma^2 = \pi^{v(a)+v''}(\gamma')^2 = \pi^{2(v'+v'')}(\gamma')^2$, where $v(\gamma') \ge 0$. Then, we rewrite \eqref{eqn:main} as 
    \begin{equation}
        \label{eqn:auxiliarygammaprime}
    \pi^{2v'+v''}a'x^2 + \pi^{2v'+v''}b'y^2 + c\pi^{2v'+v''}\pi^{v''}z^2 = 0.
    \end{equation}
    Dividing the equation by $\pi^{v(a)} = \pi^{2v'+v''}$, we note that $(\alpha, \beta, \gamma') \in \cO_{K_v}^3$ is a solution to \eqref{eqn:auxiliarygammaprime}. Conversely, if \eqref{eqn:auxiliarygammaprime} has a solution $(\alpha, \beta, \gamma) \in \cO_{K_v}^3$, it is elementary to see that $(\alpha, \beta, \pi^{v'+v''}\gamma) \in \cO_{K_v}^3$ is a solution to \eqref{eqn:main}, thereby proving the lemma.
\end{proof}

\begin{remark}
    
    Note that case (3) in the previous lemma can be treated directly by applying Corollary \ref{cor:conditions} (where congruences now should be taken over the appropiate local field, rather than over $\cO_K$). Similarly, case (1) can also be reduced to a previously solved case.

    We have preferred to avoid incorporating these complicated cases in Corollary \ref{cor:conditions}, but we remark that this gives a complete set of conditions to check, even if $a\cO_K$, $b\cO_K$ and $c\cO_K$ have common factors.
\end{remark}

\section{A parametric formula to find all solutions}
\label{Sec:parametric}
If the conditions in Corollary \ref{cor:conditions} or Lemma \ref{lemma:complicated} are satisfied, we know that \eqref{eqn:main} has at least one solution over $K$. In this case, we shall show that it has infinitely many solutions and that, furthermore, they can all be parameterised explicitly. In this section, we shall generalise the slope method used over $\Q$ (see \cite[Section 6.3.2]{Cohen}) in order to obtain these parameterisations.

\begin{proposition}
\label{prop:parameterisation}
    Let $K$ be a number field and let $(\alpha_0, \beta_0, \gamma_0) \in K^3$ be a non-trivial solution of \eqref{eqn:main} with $\gamma_0 \neq 0$. Then, \eqref{eqn:main} has infinitely many non-trivial solutions $(x,y,z)$ with $z \neq 0$ and they can be parameterised with the following formulae:
    \[\begin{cases}
        x = \pm d((bm^2 - an^2) \alpha_0 - 2bmn\beta_0), \\
        y = \pm d(-2amn\alpha_0 + \beta_0 (an^2-bm^2)), \\
        z = \pm d\gamma_0 (an^2+bm^2),
    \end{cases}
    \]
    where $m, n \in \cO_K$, and $d \in K^*$.
\end{proposition}

\begin{remark}
    The assumption that $\gamma_0 \neq 0$ is not restrictive. Indeed, since the solution $(\alpha_0, \beta_0, \gamma_0)$ is non-trivial, at least one of $\alpha_0$, $\beta_0$ and $\gamma_0$ is non-zero. By renaming the monomials $ax$, $by$ and $cz$ as necessary, we may assume that $\gamma_0 \neq 0$, as desired. Similarly, it may be necessary to change the order of $\alpha_0$, $\beta_0$ and $\gamma_0$ in order to find solutions with $z \neq 0$.     
\end{remark}

\begin{proof}
    Firstly, we begin by characterising all solutions of \eqref{eqn:main} with $z = 1$. For this purpose, let $(\alpha,\beta,1) \in K^3$ be a solution of \eqref{eqn:main} and let $t \in K$ be such that 
    \[\beta - \frac{\beta_0}{\gamma_0} = t\left(\alpha-\frac{\alpha_0}{\gamma_0}\right),\]
    or, alternatively,
    \[\beta = \frac{\beta_0}{\gamma_0} + t\left(\alpha-\frac{\alpha_0}{\gamma_0}\right).\]
    Note that the previous expression is the family of all lines going through the point $\left(\frac{\alpha_0}{\gamma_0}, \frac{\beta_0}{\gamma_0}\right)$ in the variables $\alpha$ and $\beta$ and parameterised by their slope $t$. Since $(\alpha,\beta,1)$ is a solution of \eqref{eqn:main}, the following system of equations is satisfied:
    \begin{equation}
        \label{eqn:auxsystem}
    \begin{cases}
      \beta = \frac{\beta_0}{\gamma_0} + t\left(\alpha-\frac{\alpha_0}{\gamma_0}\right), \\
      a\alpha^2 + b\beta^2 + c = 0.     
    \end{cases}
    \end{equation}
    As $t$ varies, solutions of the previous system correspond to all solutions of \eqref{eqn:main} with $z=1$. Indeed, apart from the solution
    \[(\alpha, \beta) = \left(\frac{\alpha_0}{\gamma_0}, \frac{\beta_0}{\gamma_0}\right)\]
    the system \eqref{eqn:auxsystem} has solutions given by
    \begin{equation*}
    \begin{cases}
        \alpha = \frac{(bt^2-a)\alpha_0-2bt\beta_0}{(bt^2+a)\gamma_0}, \\
        \beta = -\frac{2at\alpha_0 + (bt^2-a)\beta_0}{(bt^2+a)\gamma_0}. \\
    \end{cases}
    \end{equation*}
    Since $K$ is the field of fractions of $\cO_K$, we can write $t = m/n$ for some $m, n \in \cO_K$, so that 
    \begin{equation*}
    \begin{cases}
        \alpha = \frac{(bm^2-an^2)\alpha_0-2bmn\beta_0}{(bm^2+an^2)\gamma_0}, \\
        \beta = -\frac{2anm\alpha_0 + (bm^2-an^2)\beta_0}{(bm^2+an^2)\gamma_0}, \\
    \end{cases}
    \end{equation*}
    In addition, we note that a general solution $(x,y,z) \in K^3$ of \eqref{eqn:main} with $z \neq 0$ can be written as $(x,y,z) = (\pm \delta\alpha, \pm \delta\beta, \pm \delta)$ for some $\delta \in K$ (where the choice of sign for each variable is independent of the choice of sign for the others). We therefore obtain the expressions in the statement of the Proposition by choosing $\delta = d(bm^2+an^2)\gamma_0$ and picking the sign independently for each of the three variables $x, y$ and $z$.
\end{proof}

\begin{remark}
    If $K$ has class number one, the proof of the previous proposition can be adapted to find all solutions over $\cO_K$, rather than $K$. Indeed, suppose that $(\alpha_0, \beta_0, \gamma_0) \in \cO_K$ and for $m, n \in \cO_K$, define 
    \[D_{m,n} = \gcd((bm^2 - an^2)\alpha_0 - 2bmn\beta_0, -2amn\alpha_0 + \beta_0(an^2-bm^2), \gamma_0(an^2+bm^2)).\]
    Then, all solutions $(x,y,z) \in \cO_K^3$ to \eqref{eqn:main} can be obtained with the formulae:
    \[\begin{cases}
        x = \pm \frac{d}{D_{m,n}}((bm^2 - an^2)\alpha_0 - 2bmn\beta_0) \\
        y = \pm \frac{d}{D_{m,n}}(-2amn\alpha_0 + \beta_0(an^2-bm^2)) \\
        z = \pm \frac{d}{D_{m,n}}\gamma_0(an^2+bm^2),
    \end{cases}
    \]
    where $m, n \in \cO_K$ and $d$ now belongs to $\cO_K$, rather than $K$. In addition, picking $d = 1$ gives all primitive solutions. We remark that the previous is not a closed formula, since, in principle, the value of $D_{m,n}$ needs to be computed for each value of $m$ and $n$. 
\end{remark}

\section{An algorithm to find one solution}
\label{Sec:algorithm}
Due to Proposition \ref{prop:parameterisation}, we know that the computational complexity of completely resolving \eqref{eqn:main} is equivalent to that of finding one solution, since we have a parameterisation for all other solutions.

As we have explained in Section \ref{Sec:background}, there are several computationally efficient algorithms to compute particular solutions over $\Q$, which often rely on the LLL algorithm and simplify the equation under consideration.


Over number fields of class number greater than one, many of these methods are not available due to the failure of unique factorisation. For instance, we see that the algorithm by Cremona and Ruskin which bypasses factorisation of the coefficients (\cite[Algorithm II]{Cremona}), as well as the method by Simon \cite{Simon} do not generalise to arbitrary number fields.

However, we can adapt the original descent method by Legendre 
in order to find solutions. For simplicity, we shall assume that \eqref{eqn:main} is written in \textbf{norm form equation}, that is:
\begin{equation}
    \label{eqn:normform}
x^2 - Ay^2 = Bz^2,
\end{equation}
where $A, B \in \cO_K$ satisfy that the prime factorisations of the ideals $A\cO_K$ and $B\cO_K$ contain no squares of any principal ideals. Note that we may turn \eqref{eqn:main} into a norm form equation by multiplying by $a$ and absorbing squares into $x$, $y$ and $z$. We note that finding a solution to \eqref{eqn:normform} is equivalent to finding an element $\alpha \in K(\sqrt{A})$ with 
\[\Norm_{K(\sqrt{A})/K}(\alpha) = B.\]
In this sense, we remark that there are algorithms for resolving the \textbf{generalised Pell equation} 
\begin{equation}
    \label{eqn:Pell}
x^2 -dy^2 = c
\end{equation}
over number fields (see for example \cite{normequations}) but they are very inefficient if the coefficients $c$ and $d$ are relatively large. Therefore, in order to find a solution to \eqref{eqn:normform}, we first need to reduce $A$ and $B$ before solving a generalised Pell equation using the existing algorithms.

For this, we present the following algorithm. Our implementation is an adaptation of \cite[Algorithm I]{Cremona}, which is itself based upon Legendre's descent method. In the presentation of the algorithm, we will let $|\cdot|$ denote the complex modulus.

\begin{algorithm} \texttt{LegendreDescent} 
\label{alg:Legendre}

    \noindent \textbf{Input:} The coefficients $A$ and $B$ of the equation $x^2 - Ay^2 = Bz^2$ and the number field $K$ over which we wish to find a non-trivial solution.

    \noindent \textbf{Output:} A tuple $(x, y, z) \in K^3$ satisfying $x^2 - Ay^2 = Bz^2$.

    \noindent \textbf{Algorithm:}
    
    \noindent \textbf{1.-} Find the subfield $K' \subseteq K$ of smallest degree where $A$ and $B$ are defined and for which the local conditions in Corollary \ref{cor:conditions} hold, and replace $K$ by $K'$. If there is no such field, the equation has no solutions over $K$.

    \noindent \textbf{2.-} If $|A| > |B|$, find a solution $(\alpha', \beta', \gamma')$ to the equation $(x')^2-B(y')^2 = A(z')^2$ using \texttt{LegendreDescent} and return $(x, y, z) = (\alpha', \gamma', \beta')$.

    \noindent \textbf{3.-} If $B \in \cO_K$ is a unit, find a solution $(\alpha_0, \beta_0)$ to the generalised Pell equation $x_0^2 -Ay_0^2 = B$, and return $(x, y, z) = (\alpha_0, \beta_0, 1)$.

    \noindent \textbf{4.-} If $A \in \cO_K$ is a unit, find a solution $(\alpha_0, \gamma_0)$ to the generalised Pell equation $x_0^2 - Bz_0^2 = A$ and return $(x, y, z) = (\alpha_0, 1, \gamma_0)$.

    \noindent \textbf{5.-} If $A = B$, find a solution $(\alpha_0, \gamma_0)$ to the generalised Pell equation $x_0^2 + z_0^2 = B$ and return $(x, y, z) = (B, \alpha_0, \gamma_0)$.

    \noindent \textbf{6.-} We recall that, after step 2, we have that $|A| \le |B|$. Let $w \in \cO_K$ be a solution to the congruence equation 
    \begin{equation}
        \label{eqn:congruencealgorithm}
    X^2 \equiv A \pmod{B\cO_K},
    \end{equation}
    with $|w|$ as small as possible. If $|w| < |B|-1$, proceed to step 7. Otherwise, solve the generalised Pell equation $x^2-Ay^2 = B$ and return the solution $(x,y,1)$.

    \noindent \textbf{7.-} Use the LLL algorithm to find a solution $(\alpha_0, \beta_0) \in \cO_K^2$ to the congruence equation 
    \begin{equation}
        \label{eqn:smallcongruence}
    x_0^2-Ay_0^2 \equiv 0 \pmod{B\cO_K},
    \end{equation}
    with $|\alpha_0|^2 + |A||\beta_0|^2$ as small as possible. We refer the reader to Remark \ref{rmk:lll} for more details on how this reduction step is performed.

    \noindent \textbf{8.-} Let $t \in \cO_K$ be the element defined by
    \begin{equation}
        \label{eqn:deft}
    t = \frac{\alpha_0^2-A\beta_0^2}{B}.
    \end{equation}
    Write the ideal $t\cO_K$ as 
    \[t\cO_K = (t_1\cO_K)(t_2\cO_K)^2,
    \]
    where there are no squares of principal ideals in the prime decomposition\footnote{ It is clear that there is always one factorisation of this form, since it suffices to take $t_1 = t$ and $t_2 = 1$. In order to check if there is another factorisation with $t_2\cO_K$ being a proper ideal of $\cO_K$, one would need to compute the prime factorisation of $t\cO_K$ and check if both the \emph{squarefree part} and the \emph{square part} are principal ideals.} of $t_1\cO_K$. By multiplying $t_1$ by the appropiate unit of $\cO_K$, we can assume that 
    \[t = t_1t_2^2.\]

    \noindent \textbf{9.-} Find a solution $(\alpha_1, \beta_1, \gamma_1)$ to the equation $x_1^2-Ay_1^2 = t_1z_1^2$ using \texttt{LegendreDescent}. Then, return 
    \begin{equation}
    \label{eqn:findsolutions}
    (x,y,z) = (\alpha_0\alpha_1 + A\beta_0\beta_1, \alpha_1\beta_0 + \alpha_0\beta_1, t_1t_2\gamma_1),
    \end{equation}
    which is a solution to the equation $x^2-Ay^2 = Bz^2$.
\end{algorithm}

Let us briefly comment on the algorithm, and on the computational improvements that it introduces. Firstly, we note that it is possible that \eqref{eqn:normform} has solutions over a field $K'$ strictly smaller than $K$. In this situation, working over $K'$ rather than $K$ can make the computation significantly easier.

In steps $3$, $4$ and $5$, we are left with instances of \eqref{eqn:normform} where no further reduction is possible. However, in these cases the resolution of the equation can be done by treating it as a generalised Pell equation, as in \eqref{eqn:Pell}. Since the coefficients are generally not too big, it is almost always feasible to resolve the equations that occur in these steps.

Then, in the remaining steps of the algorithm, we carry out the reduction step. We note that, since the equation has passed the local solubility test given by Corollary \ref{cor:conditions}, we know that the congruence equation \eqref{eqn:congruencealgorithm} has a solution. However, it is not always possible to attain a solution with $|w| < |B|-1$. For instance, consider the equation
\[x^2 - 2y^2 = 3z^2,\]
to be solved over $K = \Q(\sqrt{14})$. In step $6$, we would need to solve the congruence equation
\[X^2 \equiv 2 \pmod{3\cO_K},\]
whose solution with smallest complex modulus is 
\[w = \sqrt{14},\]
which clearly does not satisfy $|w| < 3-1 = 2$. In these cases, it is impossible to reduce \eqref{eqn:normform} further, and we need to solve the corresponding equation with the methods outlined in \cite{normequations}. Otherwise, and by the definition of $t$ in \eqref{eqn:deft}, it follows that:
\begin{equation}
    \label{eqn:conditiont}
|t| \le \frac{|\alpha_0|^2 + |A||\beta_0|^2}{|B|} \le \frac{|w|^2 + |A|}{|B|} < |B|,
\end{equation}
where the middle inequality uses the fact that both $(\alpha,\beta)$ and $(w,1)$ are solutions to \eqref{eqn:smallcongruence}, with $|\alpha|^2 + |A||\beta|^2 \le |\omega|^2 + |A|\cdot 1^2$. Consequently, \eqref{eqn:conditiont} shows that the equation that we obtain in Step 9 has smaller coefficients than our original equation. In practice, it is very common for the condition $|w| < |B| - 1$ to be fulfilled and so the resulting equations are often simpler.

\begin{remark}
    \label{rmk:lll}
    We give more details on how the LLL reduction algorithm is used in order to find a small solution to the congruence equation \eqref{eqn:smallcongruence}. Firstly, let $n$ be the degree of the number field $K$ and consider a basis $\{1, \beta_2, \dots, \beta_n\}$ for $\cO_K$, so that 
    \[\cO_K = \Z[1, \beta_2, \dots, \beta_n],
    \]
    and consider the $2n-$dimensional lattice given by 
    \[L = \cO_K^2 = \left(\Z[1, \beta_2, \dots, \beta_n]\right)^2 \cong \Z^{2n}.\]
    Let $w$ be a solution to the congruence equation \eqref{eqn:congruencealgorithm} and consider the following positive definite quadratic form on $L$:
    \begin{equation}
        \label{eqn:quadraticform}
    q(u, v) = |wu+Bv|^2 + |A||u|^2.
    \end{equation}
    Let $||\cdot||$ be the norm induced by the quadratic form \eqref{eqn:quadraticform}. By using the LLL algorithm, we can find a non-zero vector $(u_0, v_0) \in L$ such that $||(u_0, v_0)||$ is small. Then, we define
    \[(x_0, y_0) = (u_0w + Bv_0, u_0).
    \]
    It is elementary to see that $(x_0, y_0) \in \cO_K^2$ satisfies the congruence equation \eqref{eqn:smallcongruence} and that 
    \[|x_0|^2 + |A||y_0|^2 = ||(u_0, v_0)||
    \]
    is as small as possible. Since the LLL algorithm only approximates the shortest vector in the lattice, there is no guarantee that this procedure will give a vector shorter than $(w, 1)$. In practice, however, we have found that the results are almost always better.

    We remark that there is some amount of research on finding short vectors of lattices over number fields. For example, we highlight the work of Fieker and Stehlé \cite{latticesNFs} and we thank them again for sharing their code with us. However, we found that the computational performance of their algorithm was generally worse for our purposes than applying LLL over $\Z$, which is why we use this algorithm in our exposition.
    
\end{remark}

Finally, we remark that out of the three main computer algebra systems generally used in mathematical research (\texttt{SAGE} \cite{Sage}, \texttt{Pari/GP} \cite{Pari} and \texttt{Magma}), only \texttt{Magma} implements an algorithm similar to our Algorithm \ref{alg:Legendre} (see \cite[Section 119.5]{handbook}), where they reduce the equation under consideration to one with smaller coefficients.

However, apart from the brief remarks in the manual, we have not been able to find a reference for the implementation of the algorithm and this is why we believe that our work here could be useful to other researchers.

We finish this section by illustrating Algorithm \ref{alg:Legendre} with an example.

\begin{example}
    Consider the equation
    \begin{equation}
        \label{eqn:example}
    x^2 - 823z^2 = -1929y^2,
    \end{equation}
    where we consider solutions over $K = \Q(\sqrt{-6})$, which has class number $2$. Firstly, we observe that \eqref{eqn:example} does not have solutions over $\Q$ since the congruence equation
    \[x^2 \equiv 823 \pmod{1929}\]
    does not have solutions. Over $K$, it can be seen that 
    \[w =  643 + 723\sqrt{-6}\]
    satisfies that 
    \[w^2 \equiv 823 \pmod{1929\cO_K}, \]
    and, in addition, $|w| < 1929-1 = 1928$. Using the LLL algorithm, we find that the pair $(\alpha, \beta) \in \cO_K^2$ satisfying \eqref{eqn:smallcongruence} with $|\alpha|^2+|A||\beta|^2$ as small as possible is given by 
    \[(\alpha, \beta) = (163+83\sqrt{-6}, -7-\sqrt{-6}),\]
    and $t$ has expression
    \[t = 26+20\sqrt{-6}.\]
    This reduces the original equation to the simpler equation
    \[x^2-(26+20\sqrt{-6})y^2 = 823z^2,\]
    which we need to solve. We do this by repeatedly applying the algorithm and reducing the equation further. All iterations of this process, including the first equation, are given in Table \ref{tab:example}. 

    \begin{table}[!ht]
        \centering
        \begin{tabular}{||ccccc||}
        \hline 
          $A$ & $B$ & $\alpha$ & $\beta$ & $t$ \\
          \hline \hline 
          $823$ & $-1929$ & $-163+83\sqrt{-6}$ & $-7-\sqrt{-6}$ & $26+20\sqrt{-6}$\\
          $26+20\sqrt{-6}$ & $823$ & $4-12\sqrt{-6}$ & $3-\sqrt{-6}$ & $-2$\\
          $-2$ & $26+20\sqrt{-6}$ & $4+2\sqrt{-6}$ & $6+\sqrt{-6}$ & $2$ \\
          $-2$ & $2$ & - & - & -  \\
    \hline 
          \end{tabular}
        \caption{Iterations of the algorithm when applied to \eqref{eqn:example}}
        \label{tab:example}
    \end{table}
    We note that, as part of the algorithm, we end up needing to solve the equation 
    \[x^2 + 2y^2 = 2z^2,\]
    which has an elementary solution $(x,y,z) = (0,1,1)$. Despite the original equation not having solutions over $\Q$, we find an equation which does have rational solutions. Then, we simply pull back the solutions of each of the equations with the expression \eqref{eqn:findsolutions}, and obtain the solution 
    \[(x,y,z) = (-108508+13308\sqrt{-6}, 3092-1644\sqrt{-6}, 1120-1268\sqrt{-6}),
    \]
    to the original equation \eqref{eqn:example}. We note that, in each of the iterations, the value of $|B|$ was reduced by more than $50\%$ each time, which quickly leads to manageable values of $A$ and $B$.
\end{example}

\section{Minimal solutions}
\label{Sec:minimal}
As we have explained in Section \ref{Sec:background}, Hölzer showed that, if \eqref{eqn:main} has solutions, it necessarily has a solution $(\alpha, \beta, \gamma)$ satisfying \eqref{eqn:Holzercondition}.
 
 Over number fields, it is unclear whether an analogous result holds, but there are some partial results. Specifically, we highlight the following theorem by Díaz-Vargas and Vargas de los Santos \cite[Theorem 6]{VargasSantos}, generalising previous work by Leal Ruperto \cite[Theorem 1]{LealRuperto} over $\Q(i)$ and by Mordell \cite{Mordell} over $\Q$.

\begin{theorem}(Díaz-Vargas and Vargas de los Santos)
    \label{thm:holzer}
    Let $d = 1, 2, 3, 7$ or $11$ and consider the number field $K = \Q(\sqrt{-d})$, with ring of integers $\cO_K$. Assume furthermore that the Diophantine equation
    \[ax^2 + by^2 + cz^2 = 0,\]
    where $a, b, c \in \cO_K$ are fixed, has a solution $(\alpha, \beta, \gamma) \in \cO_K^3$. Then, it has a solution $(\alpha_0, \beta_0, \gamma_0) \in \cO_K^3$ with either
    \begin{equation}
        \label{eqn:holzer1}
    |\gamma_0| \le \sqrt{\frac{4}{3-d}|ab|} \quad \text{if } d = 1, 2,
    \end{equation}
    or
    \begin{equation}
        \label{eqn:holzer2}
    |\gamma_0| \le  \sqrt{\frac{16d}{-d^2+14d-1}|ab|} \quad \text{if } d=3,7,11. 
    \end{equation}
\end{theorem}

\begin{remark}
    We note that all the number fields considered in Theorem \ref{thm:holzer} are Euclidean. This ingredient is vital in the proof of the theorem and justifies why, to the best of our knowledge, there are no analogous results for number fields of class number different than one. We also remark that Theorem \ref{thm:holzer} does not say anything about the size of $|\alpha_0|$ and $|\beta_0|$ and is, in this sense, less general than the Hölzer condition over $\Q$.
\end{remark}

Theorem \ref{thm:holzer} guarantees that, if \eqref{eqn:main} has solutions, there are solutions which are ``minimal'' in some sense. While the proof of \cite[Theorem 6]{VargasSantos} is not algorithmic in nature, it is constructive, and we may adapt it to obtain an algorithm which allows us to obtain a solution satisfying conditions \eqref{eqn:holzer1} or \eqref{eqn:holzer2} from a particular solution $(\alpha_0, \beta_0, \gamma_0) \in \cO_K^3$ to \eqref{eqn:main}. Our proposed algorithm is the following:

\begin{algorithm} \texttt{ReduceSolution}
\label{alg:ReduceSolution}

    \noindent \textbf{Input:} The number field $K=\Q(\sqrt{-d})$ (where $d = 1, 2, 3, 7$ or $11$), the coefficients $a, b$ and $c$ of $ax^2 + by^2 + cz^2 = 0$, 
    and a particular solution to the equation, $(\alpha_0, \beta_0, \gamma_0) \in \cO_K^3$

    \noindent \textbf{Output:} A tuple $(x,y,z) \in \cO_K^3$ satisfying $ax^2 + by^2 + cz^2 = 0$ and either \eqref{eqn:holzer1} or \eqref{eqn:holzer2}, depending on the value of $d$.

    \noindent \textbf{Algorithm:}
    
    \noindent \textbf{1.-} Compute $d_0 = \gcd(\alpha_0, \beta_0, \gamma_0) \in \cO_K$. If $d_0 \neq 1$, replace the tuple $(\alpha_0, \beta_0, \gamma_0)$ by $\left({\alpha_0}/{d_0}, {\beta_0}/{d_0}, {\gamma_0}/{d_0}\right)$.

    \noindent \textbf{2.-} By using the Euclidean algorithm, compute a solution $(\alpha, \beta) \in \cO_K$ to the equation 
    \[\beta_0 X - \alpha_0 Y = c.\]
    Then, define $\gamma$ to be the closest element of $\cO_K$ to 
    \[-\frac{a\alpha_0\alpha+b\beta_0\beta}{c\gamma_0}.\]

    \noindent \textbf{3.-} Let $(x,y,z)$ be given by 
    \[x = \frac{1}{c}(\alpha_0(a\alpha^2+b\beta^2+c\gamma^2)-2\alpha(a\alpha_0\alpha+b\beta_0\beta+c\gamma_0\gamma)), \]
    \[y = \frac{1}{c}(\beta_0(a\alpha^2+b\beta^2+c\gamma^2)-2\beta(a\alpha_0\alpha+b\beta_0\beta+c\gamma_0\gamma)),\]
    and
    \[z = \frac{1}{c}(\gamma_0(a\alpha^2+b\beta^2+c\gamma^2)-2\gamma(a\alpha_0\alpha+b\beta_0\beta+c\gamma_0\gamma)).\]
    If $|z|$ satisfies the bounds \eqref{eqn:holzer1} or \eqref{eqn:holzer2}, return $(x,y,z)$. Otherwise, replace $(\alpha_0, \beta_0, \gamma_0)$ by $(x, y, z)$ and go back to Step 1.    
\end{algorithm}

The correctness of this algorithm is essentially in the proof of \cite[Theorem 6]{VargasSantos}, since the values of $(x,y,z)$ computed in Step 3 are solutions of \eqref{eqn:main} and satisfy that $|z| < |\gamma_0|$. Since, for any number field $K$ and any $B\in \mathbb{R}$, the set 
\[A_B = \{\alpha \in \cO_K \mid |\alpha| \le B\} \]
is finite, we are guaranteed to find a value of $z$ satisfying \eqref{eqn:holzer1} or \eqref{eqn:holzer2} in a finite number of steps.

\section{Concluding remarks and future lines of work}
\label{Sec:conclusions}
In this work, we have shown that many of the results for the generalised Pythagorean equation \eqref{eqn:generalisedPythagorean} hold over number fields $K$. However, our understanding of them is still not as complete as in the case $K =\Q$. 

In particular, we find that the main difficulties are both computational and arising from the minimality condition in Theorem \ref{thm:holzer}. In this sense, we think that the following are the most natural problems to consider in this area.

\textbf{Problem I}: Our implementation of Algorithm \ref{alg:Legendre} requires the factorisation of $B\cO_K$ in order to solve the congruence equation \eqref{eqn:congruencealgorithm}. For $K=\Q$, Cremona and Rusin have developed a more efficient algorithm which bypasses this factorisation (see \cite[Section 2.5, Algorithm II]{Cremona}). Is it possible to develop a similar algorithm for a general number field $K$?

\textbf{Problem II}: As remarked in Section \ref{Sec:minimal}, the minimality conditions \eqref{eqn:holzer1} and \eqref{eqn:holzer2} measure only the size of $|z|$, while the Hölzer condition over $\Q$ controls simultaneously the size of the three variables. Is it possible to do this for any of the number fields $K$ in Theorem \ref{thm:holzer}?

\textbf{Problem III}: Can Theorem \ref{thm:holzer} be proved for number fields $K$ which are not Euclidean imaginary quadratic? If so, can an algorithm similar to Algorithm \ref{alg:ReduceSolution} be developed?

\end{document}